\documentclass[11pt]{article}
\usepackage{amsmath}
\usepackage{amsthm}
\usepackage{amssymb}

\usepackage{fullpage}
\usepackage{color,graphics}
\usepackage{enumerate}
\usepackage{graphicx}

\theoremstyle{plain}                    %------- 'regular' theorem types

\newtheorem{theorem}{Theorem}[section]
\newtheorem{lemma}[theorem]{Lemma}

\newtheorem{corollary}[theorem]{Corollary}

\theoremstyle{definition}

\theoremstyle{remark}
\newtheorem{remark}[theorem]{Remark}

\newcommand{\p}{\mathbb{P}}
\newcommand{\e}{\mathbb{E}}

\numberwithin{equation}{section}

\title{On the evolution in the configuration model
\thanks{Keywords: random regular graphs, configuration model.
MSC classification: 05C80, 60J10}}

\author{Ton\'{c}i Antunovi\'{c} \\ University of California, Los Angeles \\ {\tt tantunovic@math.ucla.edu}}

\date{}

\begin{document}
\maketitle

\begin{abstract}
We give precise estimates on the number of active/inactive half-edges in the configuration model used to generate random regular graphs. This is obtained by analyzing a more general urn model with negative eigenvalues.
\end{abstract}
%\end{titlepage}
%\setcounter{page}{1}
%}{}

\section{Introduction and statements or results}

Random $d$-regular graph, a random graph obtained by sampling uniformly a $d$-regular vertex labeled simple graph on $n$ vertices (for even $nd$), is an important random graph model. While for $d=2$ the resulting graph is a union of disjoint cycles (and in particular disconnected with high probability), for $d \geq 3$ the random regular graph, with probability converging to 1 (as $n \to \infty$) is connected, and even has Hamiltonian cycles. Furthermore, for $d \geq 3$ it looks fundamentally different from a $d$-dimensional torus, in particular it is an expander, locally tree-like and has bounded number of cycles of fixed length (for more properties, proofs and references to the original work see \cite{Bollobas01, JLR00}). One of the main tools used for both generating and studying random regular graphs is the configuration model introduced by Bollob\'as in \cite{Bollobas80}. This is an algorithm which can be described as follows. Start with $n$ disconnected labeled vertices each equipped with $d$ half-edges. Select an  arbitrary vertex and mark it and it's half-edges as active, and label all other vertices and half-edges as inactive. At each stage of the algorithm choose in some way a labeled half-edge $e$ and connect it with a uniformly selected half-edge $f$ (active or not), and mark them both as used. If $f$ corresponds to an inactive vertex, also label this vertex and all its half-edges (except $f$) as active. When all half-edges are connected, collapse the connected half-edges to edges. The algorithm generates a random graph which might not be simple if two half-edges incident to the same vertex get connected, or there are two connections of half-edges incident to the same pair of vertices. However, conditioned on the event that the random graph is simple, it has the distribution of the random regular graph. Moreover, the probability that the graph produced by the configuration model is simple, is bounded away from zero as $n \to \infty$ and $d$ is fixed (see \cite{Bollobas80}, and an earlier related work \cite{BenderCanfield}).

Denote by $A_n$ and $I_n$ the number of active and inactive half-edges after the $n$-th step respectively (the number of used half-edges is clearly $2n$). Then it is easy to see that when $I_n \geq 1$, $A_n \geq 1$ 
\[\p(I_{n+1}=I_n-d, A_{n+1}=A_n+d-2) = \frac{I_n}{A_n+I_n-1}\]
and
\[
\p(I_{n+1}=I_n,A_{n+1}=A_n-2) = \frac{A_n-1}{A_n+I_n-1}. 
\] 
Therefore the process $(I_n,A_n-1)_n$ is an urn model with the replacement matrix
\[
A = \left(\begin{array}{rr} -d & d-2 \\ 0 &
  -2 \end{array} \right)~.
\]
In the present paper we will give precise estimates on the behavior of the process $(I_n,A_n)$. The results will be indispensable for the analysis of a competing multi-type version of a first passage percolation process on random regular graphs performed in \cite{main}.
To analyze the process $(I_n,A_n)$ we will consider a more general urn model with the replacement matrix
\begin{equation}\label{eq:urn_matrix}
A = \left(\begin{array}{rr} -b & b-a \\ 0 &
  -a \end{array} \right)~,
\end{equation}
for real valued $0<a < b$. The corresponding urn process $(X_n,Y_n)_n$ (which has the same distribution as $(I_n,A_n-1)_n$ for $a=2$ and $b=d$) is a Markov chain with the transition probabilities
\begin{align}\label{eq:transition_probabilites}
\p(X_{n+1}=X_n-b, Y_{n+1} = Y_n+b-a) & = \frac{X_n}{X_n+Y_n} \nonumber \\ \p(X_{n+1}=X_n, Y_{n+1} = Y_n-a) &= \frac{Y_n}{X_n+Y_n}.
\end{align}
The process stops at a random time $\rho$ defined as the first time $n$ such that $X_n<0$ or $Y_n < 0$ or $X_n+Y_n = M-an \leq 0$, where $M=X_0+Y_0$ ($X_n=X_\rho$, $Y_n=Y_\rho$ for $n \geq \rho$). Observe that $\rho$ is a stopping time with respect to the natural filtration $\mathcal{F}_n$ induced by the outcomes of the first $n$ draws.

When $a$ and $b$ are integers it can be described as follows. Start with $X_0$ blue and $Y_0$ red balls in an urn. Draw a ball from the urn uniformly at random. If the drawn ball is blue return it into the urn and then add $b-a$ red balls into the urn and remove $b$ blue balls.
If the drawn ball is red return it into the urn and remove $a$ red balls from the urn. Stop when the urn contains no more balls, or when the number of balls of some type in the urn becomes ``negative''. Then $X_n$ and $Y_n$ stand for the number of blue and red balls in the urn after $n$ draws. In the discussion that follows we will use ``the ball terminology'' even when dealing with non-integer values of $a$ and $b$.

Since the number of blue balls $X_n$ can decrease only by $b$, we will always assume that $X_0$ is a multiple of $b$ (this in particular holds in the configuration model where $X_0  = (n-1)d$ is divisible by $d$). Having this assumption, the value of $X_n$ can not become negative, and then $\rho$ is the smallest integer $n$ such that either  $Y_n < 0$ or $X_n + Y_n = M-an \leq 0$. % in other words $X_n = X_\rho$ and $Y_n =Y_\rho$ for all $n \geq \rho$. 
%Note that $\rho$ is a stopping time with respect to the natural filtration $\mathcal{F}_n$ induced by the outcomes of the first $n$ draws. 
Observe that at each step the sum $X_n+Y_n$ decreases by $a$ (we remove exactly $a$ balls), that is $X_n + Y_n = M-an$, so a natural assumption could be that $M= X_0 +Y_0$ is a multiple of $a$. However, as Corollary \ref{cor:estimates_on_stopping} shows with probability converging to 1 (as  $M \to \infty$  and $a$ and $b$ stay fixed) blue balls indeed get exhausted before the red ones, after which the process is deterministic and consists of removing the leftover red balls. Therefore, with probability converging to 1 we have $M-a\rho \leq 0$ and then the assumption that $M$ is a multiple of $a$ only affects the final number of balls $M-a\rho$ by being either equal to zero or negative. 
For these reasons we will not assume that $M$ is a multiple of $a$.
The configuration model does not satisfy the assumption that $M=nd-1$ is a multiple of $a=2$ anyway. 
Furthermore, the discussion above shows that the process is most interesting when the initial number of blue balls $X_0$ is large.

Using the notation and terminology described above we state our results.
Our first result concerns with the number of blue balls in the urn. It provides estimates on the time when the blue balls are exhausted, as well as the number of balls except shortly before the exhaustion time. In particular it shows that blue blue balls get exhausted when the number of leftover red balls is  $O(1)MX_0^{-a/b}$. Furthermore, as long as the number of blue balls $X_n$ is large, it behaves as $X_n = (1\pm o(1))X_0(1-an/M)^{b/a}$. 
%Note that \eqref{eq:concentration in urn models_1} holds in a more general setting, see Remark \ref{rem:generalization to eigenvalues}.

In the rest of the paper we use the notation $C=C(a,b)$ for a strictly positive finite constant $C$ which only depends on the values of parameters $a$ and $b$.

\begin{theorem}
\label{thm:active_edges}
%Assume $0 < a < b$ are positive integers. Let $(X_n,Y_n)$ be the
% urn process with replacement matrix $
%\left(\begin{array}{rr} -b & b-a\\ 0 & -a
%\end{array}\right)
%$ and let $M= X_0 + Y_0$. 
Consider the process
\[
K_{n}= \frac{X_n}{X_0(1-an/M)^{b/a}}
%, \ \text{ and } \ L_{n}=
%\frac{S_n}{(M-an)-Z_0(1-an/M)^{b/a}},
\]
and the stopping time $\tau_{X_0,M} \leq \infty$ as the smallest integer $n$
such that $X_n=0$.  For fixed $\varepsilon < 1/2$ and $t>0$ define the integer
\[
n_t = \lfloor M(1-tX_0^{-a/b})/a\rfloor.
\]
and the event
\[
\mathbf{K}_{X_0,M,t,\varepsilon}= \{\rho > n_t \text{ and } |K_{n} - 1| \leq \varepsilon \text{
  for all } 0 \leq n \leq n_{t}\}.
\]
Then there exists a constant $C=C(a,b)$ such that 
\begin{equation}\label{eq:concentration in urn models_1}
\mathbb{P}(\mathbf{K}_{X_0,M,t,\varepsilon}) \geq 1-
\frac{C}{t^{b/a}\varepsilon^2}, \ \text{for all } \ t \geq \frac{CX_0^{a/b}}{M\varepsilon},% \ ( \Leftrightarrow M-an_t \geq C/\epsilon),
\end{equation}
and
\begin{equation}\label{eq:concentration in urn models_2}
\mathbb{P}(\tau_{X_0,M}\geq  n_{t}) \leq Ct^{b/(2b-a)}, \ \text{for all } \ t \geq \frac{bX_0^{a/b}}{M}. %\ ( \Leftrightarrow M-an_t \geq 1).
\end{equation}
\end{theorem}

%Note that the conditions on $t$ in \eqref{eq:concentration in urn models_1} and \eqref{eq:concentration in urn models_2} can be writen (up to the change of the constant $C$) as $M-an_t \geq C/\epsilon$ and $M-an_t \geq 1$.
The above estimates can be stated without introducing the variable $t$.
The estimate in \eqref{eq:concentration in urn models_1} can be written as follows: for any $n$ such that $M-an \geq C/\epsilon$ we have both $\rho > n$ and $|K_k - 1| \leq \epsilon$ for all $0 \leq k \leq n$ with probability at least
\begin{equation}\label{eq:concentration in urn models_1-1}
1 - \frac{C}{X_0(1-an/M)^{b/a}\epsilon^2}.
\end{equation}
The estimate in \eqref{eq:concentration in urn models_2} can be written as follows: for any $n$ such that $M-an \geq b$ we have 
\begin{equation}\label{eq:concentration in urn models_1-2}
\p(\tau_{X_0, M} \geq n) \leq CX_0^{a/(2b-a)}(1-an/M)^{b/(2b-a)}.
\end{equation}
In particular the estimate \eqref{eq:concentration in urn models_1-2} gives the bound for all relevant values of $n$.

The following two corollaries are straightforward from Theorem \ref{thm:active_edges}.

\begin{corollary}\label{cor:active_edges}
For a positive real number $m$  define $\sigma_{m}$ as the first time $n$ such that $X_n \leq m$. There is a constant $C=C(a,b)$ such that for every $\epsilon < 1/2$ with probability at least $1 - \frac{C}{m\varepsilon^{2}}$ hold both $\rho > \sigma_m$ and that 
\[
(1-\epsilon)X_0 \Big(1-\frac{ak}{M}\Big)^{b/a}\leq X_k \leq (1+\epsilon)X_0 \Big(1-\frac{ak}{M}\Big)^{b/a}, \ \text{ for all } k \leq \sigma_m.
\]
\end{corollary}

% \begin{proof}
% By increasing the value of the constant $C$ we can assume $m \geq 1$, and $M$ to be large enough so that for $n$ defined as the largest integer such that $2X_0(1-an/M)^{b/a} \leq  m$ satisfies $M-an>a$ and in particular $X_0(1-an/M)^{b/a} \geq 2^{-b/a-1}m$. Now the claim follows from the estimate in \eqref{eq:concentration in urn models_1-1}.
% \end{proof}

Denote by $\mathbf{R}$ the event that the  process doesn't end before blue balls are exhausted, that is $\mathbf{R} = \{M-a\rho \leq 0\}$.

\begin{corollary}\label{cor:estimates_on_stopping}
There exist a positive constant $C=C(a,b)$ such that
% the process ends before blue balls are exhausted (that is $M-a\rho > 0$) with probability at most 
\[
\p(\mathbf{R}) \geq 1 -\frac{CX_0^{a/(2b-a)}}{M^{b/(2b-a)}} \geq 1- \frac{C}{M^{(b-a)/(2b-a)}}.
\]
\end{corollary}

% \begin{proof}
% This straightforward by taking $n$ to be the largest integer such that $M-an < 1+a$ in \eqref{eq:concentration in urn models_1-2}.
% \end{proof}

The second result gives the estimate on the number of red balls $Y_n$. As the previous theorem shows it is close to
\[
(M-an)-X_0\Big(1 - \frac{an}{M}\Big)^{b/a}.
\]
The following result shows that this estimate holds throughout the process life-time and, as in the previous theorem does not depend on the starting configuration. 

\begin{theorem}\label{thm:inactive_edges}
Consider the process
\[
L_n = \frac{Y_n}{(M-an)-X_0(1-an/M)^{b/a}},
\]
and the event
\[
\mathbf{L}_{X_0,M,\varepsilon} = \left\{|L_n-1| \leq \varepsilon, \ \text{for all } 0 \leq n < M/a\right\}.
\]
For any $\varepsilon > 0$ there is a positive sequence
  $(\lambda_{\varepsilon,M})_M$  converging to $0$ as $M \to \infty$ such that for all starting configurations $0
  < X_0 < M$
 \[
\mathbb{P}(\mathbf{L}_{X_0,M,\varepsilon})
\geq 1-\lambda_{\varepsilon,M}.
\]
\end{theorem}

Both of the above theorems can be applied to the configuration model.
See Figure \ref{figure} for the ratio of active and inactive half-edges in the configuration model used to generate the random regular graph of degree $d=20$ as predicted by the above theorems.
 In particular they imply that when inactive vertices get exhausted (and all the vertices get connected) there are still roughly $n^{1-2/d}$ active half-edges left.  While to generate random regular graphs one usually runs  the configuration model  with $X_0=d(n-1)$, $Y_0=d$, the urn model with general positive integer $X_0$ and $Y_0$ corresponds to a configuration model in which the graph is partially constructed in the beginning. This generalization is relevant for the analysis of a competing first passage percolation model in \cite{main}.
 \begin{figure}\label{figure}
\begin{center}
\includegraphics[scale=0.3]{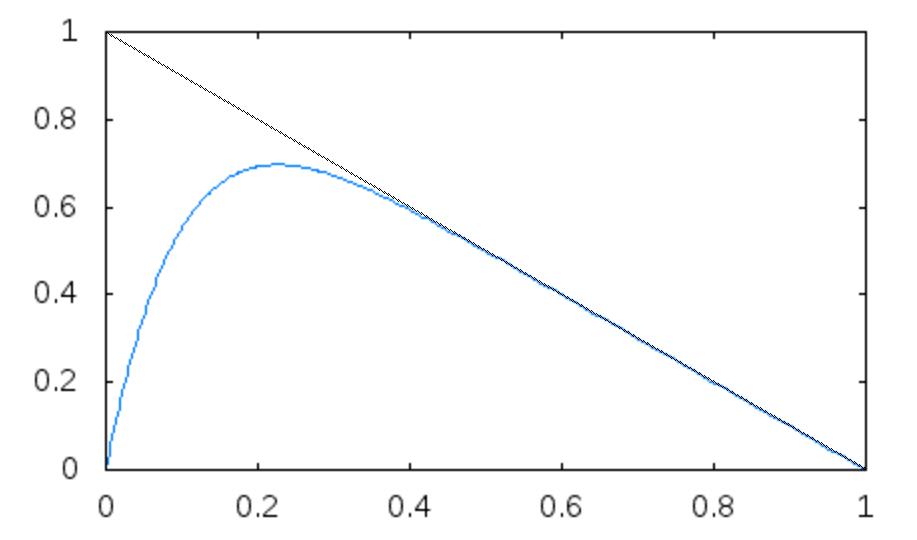}
\caption{Proportion of active half-edges ($A_k/nd$) in the configuration model for $d=20$ as predicted by theorems \ref{thm:active_edges} and \ref{thm:inactive_edges} represented by the blue line. Horizontal axis represents rescaled values of $2k/nd \in [0,1]$. Diagonal line represents the number of unused half-edges $A_k + I_k=nd-2k$.}
\end{center}
 \end{figure}
\section{Proofs}

First we give an elementary technical estimate.

\begin{lemma}\label{lemma:technical}
For any positive integer $n$ such that $M-an \geq 2b$ we have
\[
e^{-\frac{b^2}{a(M-an)}}\Big(1-\frac{an}{M}\Big)^{b/a}\leq \prod_{k=0}^{n-1}\Big(1-\frac{b}{M-ak}\Big) \leq e^{\frac{b}{M-an}}\Big(1-\frac{an}{M}\Big)^{b/a}.
\]
\end{lemma}

\begin{proof}
Denote 
\[
S_n = \sum_{k=0}^{n-1}\log\Big(1- \frac{b}{M-ak}\Big).
\]
For $0 \leq t \leq 1/2$ it holds that  $-t -t^2 \leq\log(1-t) \leq -t$ and so we have  
\[
- \sum_{k=0}^{n-1}\frac{b}{M-ak} - \sum_{k=0}^{n-1} \frac{b^2}{(M-ak)^2}\leq 
S_n
\leq - \sum_{k=0}^{n-1}\frac{b}{M-ak}.
\]
Comparing the sum and the integral 
\[
\frac{b}{a}\int_{M-an}^{M}\frac{dt}{t} - \frac{b}{M-an} + \frac{b}{M}\leq
\sum_{k=0}^{n-1}\frac{b}{M-ak}
\leq
\frac{b}{a}\int_{M-an}^M\frac{dt}{t},
\]
gives
\[
\frac{b}{a} \log\Big(1-\frac{an}{M}\Big) - \frac{b^2}{a(M-an)}
\leq
S_n
\leq
\frac{b}{a} \log\Big(1-\frac{an}{M}\Big) + \frac{b}{M-an},
\]
which implies the claim.
% Again using the Taylor expansion of $\log(1-t)$ we have
% \[
% \log\Big(1-\frac{an}{M+a}\Big) \leq - \frac{an}{M+a}, \ \log\Big(1-\frac{an}{M}\Big) \geq -\frac{an}{M} - \frac{(an)^2}{2M^2},
% \]
% which easily implies
% \[
% \log\Big(1-\frac{an}{M+a}\Big) \leq \log\Big(1-\frac{an}{M}\Big)  + \frac{(an)^2}{2M^2} + \Big(\frac{an}{M} - \frac{an}{M+a}\Big).
% \]
% Combined with \eqref{eq:auxiliary_estimate} this implies the claim.

\end{proof}

\begin{remark}\label{rem:Upper_bound_on_stopping-size}
While by definition $M-an_t \geq MtX_0^{-a/b}$ the condition $t\geq aX_0^{-a/b}/M$, which is satisfied for all $t$ appearing in \eqref{eq:concentration in urn models_1} and \eqref{eq:concentration in urn models_2} implies $M-an_t \leq 2MtX_0^{-a/b}$. This will be used in estimates below.
\end{remark}

\begin{proof}[Proof of Theorem \ref{thm:active_edges}, inequality \eqref{eq:concentration in urn models_1}] 
Recall the definitions of the $\sigma$-algebra $\mathcal{F}_n$ and the stopping time $\rho$, and consider the random variable $\xi_n = \mathbf{1}_{\{\rho>n\}}$ which is the indicator of the event that $\rho>n$. Since $X_n + Y_n = M-a(n\wedge \rho)$, for all $n$,   the assumption $M-a\rho >0$ implies that $M-a\rho < X_\rho < M-a\rho +a$. 
%If $\rho < \infty$, the urn process, by definition, will be killed at time $\rho$ so $X_n = X_\rho$ and $Y_n = Y_\rho$ for all $n > \rho$.
It is an easy computation that
\[
\e[X_{n+1}|\mathcal{F}_n] = \Big(1- \frac{b\xi_n}{M-an}\Big)X_n,
\] 
which shows that the process 
\begin{equation}\label{eq:definition_of_m}
M_n = \frac{X_n}{X_0} \prod_{k=0}^{n-1} \Big(1-\frac{b\xi_k}{M-ak}\Big)^{-1} = \frac{X_{n \wedge \rho}}{X_0} \prod_{k=0}^{(n \wedge \rho) -1} \Big(1-\frac{b}{M-ak}\Big)^{-1}.
\end{equation}
is a martingale with $M_0=1$. To estimate the variance observe that for $\rho > n$ we have
\begin{multline}
X_0^{2}\prod_{k=0}^{n} \Big(1-\frac{b\xi_k}{M-ak}\Big)^{2}\e[(M_{n+1}-M_n)^2|\mathcal{F}_n]  \\
=\Big(X_n-b - X_n\big(1-\frac{b\xi_n}{M-an}\big) \Big)^2\frac{X_n}{X_n+Y_n} + \Big(X_n - X_n\big(1- \frac{b\xi_n}{M-an} \big)\Big)^2\frac{Y_n}{X_n+Y_n}
\\ = \frac{b^2X_n}{M-an} - \frac{b^2X_n^2\xi_n}{(M-an)^2} \leq \frac{b^2X_n}{M-an},
\end{multline}
while for $\rho \leq n$ the left hand side is equal to $0$.
In any case for $M-an \geq 2b$ we have 
\[
\e[(M_{n+1}-M_n)^2] \leq \frac{b^2\e[M_n]}{(M-an-b)X_0\prod_{k=0}^{n}(1-b/(M-ak))},
\]
and by
 Lemma \ref{lemma:technical} there is a constant $C_1=C_1(a,b)$ such that
\[
\e[(M_{n+1}-M_n)^2] \leq \frac{C_1}{(M-an)X_0(1-an/M)^{b/a}}.
\]
Since by the assumptions
\[
M-an_t \geq tMX_0^{-a/b} \geq C/\epsilon \geq 2C,
\]
choosing the value $C$ in the statement larger than $b$, and
using the fact that $M_n$ is a martingale yields
\[
\e[(M_{n_t}-1)^2] \leq 
\frac{C_1 M^{b/a}}{X_0}\sum_{k=0}^{n_t-1} \frac{1}{(M-ak)^{b/a+1}}
\leq
\frac{C_1M^{b/a}}{bX_0(M-an_{t})^{b/a}} \leq
\frac{C_1}{bt^{b/a}}.
\]
% \[
% \sum_{k=0}^{n_{1,t}-1} r_kq_{0,k-1} \leq C_4
% \frac{C_3M^{b/a}}{Z_0}\sum_{k=0}^{n_{1,t}-1} \frac{1}{(M-ak)^{b/a+1}}
% \leq \frac{C_5M^{b/a}}{Z_0(M-an_{1,t})^{b/a}} \leq
% \frac{C_5}{t^{b/a}},
% \]
% for some $C_5=C_5(a,b)$,
% Defining $I_0=K'_0=1$ and $I_n = K'_n -
% \sum_{k=0}^{n-1}\mathbb{E}(K'_{k+1}-K'_{k}|\mathcal{F}_k)$, as in
% Lemma \ref{lemma: estimates for martingales} we have that
% \[
% \mathbb{E}((I_{n_{1,t}}-1)^2) \leq \frac{C_5}{t^{b/a}},
% \]
% and by a Doob's maximal theorem we have that for any $\delta > 0$
% \[
% \mathbb{P}(\max_{0 \leq n \leq n_{1,t}}|I_n-1| \geq \delta) \leq
% \delta^{-2}\mathbb{E}(\max_{0 \leq n \leq n_{1,t}}|I_n-1|^2) \leq
% \frac{2C_5}{t^{b/a}\delta^2}.
% \]
% The assumptions on $n_t$ now imply the statement.
Combining this with the Doob's maximal inequality we have that the event $|M_n-1| \leq \varepsilon/2 $ for all $n \leq n_t$ has probability at least
\[
1- 2\frac{\e[(M_{n_t}-1)^2]}{(\varepsilon/2)^2} \geq 1- \frac{8C_1}{bt^{b/a}}.
\]

Now we need to replace the process $M_n$ with the process $K_n$ and account for the event $\rho < n_t$. First observe that the above bound on the process extends  to the process $K_n' = K_{n\wedge \rho}$. This follows directly from the second equality in \eqref{eq:definition_of_m} and Lemma \ref{lemma:technical},  using the inequality $M-an_t \geq C/\epsilon \geq 2C$.

To further replace the process $K_n'$ with $K_n$ and account for the event $\rho < n_t$, it suffices  to show that  $\p(\rho < n_t) \leq \frac{C_2}{t^{b/a}\epsilon^2}$, for a  constant $C_2=C_2(a,b)$, and the rest of the proof of \eqref{eq:concentration in urn models_1} is devoted to this. Note that without loss of generality we can assume that $t \leq X_0^{a/b}$, otherwise  $n_t <0$. This assumption in turn implies $t^{-b/a} \geq X_0^{-1} \geq M^{-1}$, and since $\epsilon < 1/2$,
 it actually suffices to prove that $\p(\rho < n_t ) \leq \frac{4C_2}{M}$, for all $\frac{CX_0^{a/b}}{M\varepsilon} \leq t \leq X_0^{a/b}$.

Now if $\rho < n_t$, the fact that $X_\rho > M-a\rho$ and the estimate $K_{n_t}' = K_{\rho} \leq 1+\epsilon$ would imply
\[
1-\frac{a\rho}{M} \geq \Big(\frac{M}{(1+\epsilon)X_0}\Big)^{a/(b-a)}.
\]
% Therefore, with probability  at least $1-Ct^{-b/a}\epsilon^{-2} \geq 1-4C/M$ we have 
% \[
% 1-\frac{a(\rho \wedge n_t)}{M} \geq \Big(\frac{M}{(1+\epsilon)X_0}\Big)^{a/(b-a)}.
% \]
By the proven estimate for the process $K_n'$
\[
\p\Big(\rho < n_t, 1-\frac{a\rho}{M} < \big(\frac{M}{(1+\epsilon)X_0}\big)^{a/(b-a)}\Big) \leq \p(K_{n_t}'> 1+\varepsilon) \leq \frac{C}{t^{b/a}\epsilon^2}.
\]
In particular this completely handles the case when $X_0 < M/(1+\epsilon)$, so from now on assume that $X_0 \geq M/(1+\epsilon)$ (that is $Y_0 \leq \epsilon M/(1+\epsilon)$). Note that in order to prove the bound \eqref{eq:concentration in urn models_1} we can  assume $\epsilon$ to be bounded from above by a constant smaller than $1/2$ (we simply need to adjust the constant $C$ to extend the bound to all $\epsilon < 1/2$). For the computations that follow it is convenient to assume that $\epsilon < b/a -1$, so we will assume this to hold until the end of the proof of \eqref{eq:concentration in urn models_1}. Let $\kappa = b/(1+\epsilon)-a > 0$.

Consider the stopping time $\tau_3$ as the first index $n$ such that $X_n  < M/(1+\epsilon)^{-1}$. We will prove that $\p(\rho < \tau_3) \leq C_3/M$ for some positive constant $C_3 =C_3(a,b)$. The analysis will be split into three parts: bounding the probabilities of the events $\{\rho < \tau_1\}$, $\{\tau_1 \leq \rho < \tau_2\}$ and $\{\tau_2 \leq \rho < \tau_3\}$ where $\tau_1 $ and $\tau_2$ are stopping times defined as the first times $n \geq 0$ such that $Y_n > a$ (that is $X_n < M-an-a$) and $Y_n > M^{1/3}$ (that is $X_n < M-an - M^{1/3}$).

First for $\rho \leq \tau_1$ assume that $Y_0 \leq a$ (so that $\tau_1 >0$) and observe that (for some positive $C_{3,1}=C_{3,1}(a,b)$) with probability at least $1-C_{3,1}/M$ initial $\lceil a/(b-a)\rceil$  steps result in drawing a blue ball. After this the number of red balls will be strictly larger than $a$, and thus this handles the case $n < \tau_1$.

Now we consider the case $\tau_1 \leq \rho < \tau_2$. Since we already analyzed the first case, by Markov property we can now assume that $Y_0 > a$, so that $\tau_1 =0$.
Moreover, assume that $Y_0 \leq M^{1/3}$, so that $\tau_2>0$. Select an integer $l_1 \geq 2$ such that $(l_1-1)(b-a) \geq a+1$ and 
define $\mathbf{A}$ as the event that for all integers $0 \leq k \leq M^{1/3}$ there is at most one red draw in the steps $l_1k+1, l_1k+2, \dots, l_1k+ l_1$.  By the definition of $l_1$ and since $Y_0 >a$, this event implies that $Y_{l_1\lfloor M^{1/3} \rfloor + l_1} > M^{1/3}$, and in particular $\rho \geq \tau_2$, so to handle this case we only need to show that the probability of $\mathbf{A}^c$ is at most $O(1/M)$. To end this observe that in this regime the value of $Y_n$ is always bounded from above by $2bl_1M^{1/3}$ and so the probability that for a fixed integer $0 \leq k \leq M^{1/3}$ more than two of the steps $l_1k+1, l_1k+2, \dots, l_1k+ l_1$ are red draws is at most
\[
1 - \Big(1- \frac{2bl_1M^{1/3}}{M}\Big)^{l_1} - l_1 \frac{2bl_1M^{1/3}}{M}\Big(1- \frac{2bl_1M^{1/3}}{M}\Big)^{l_1-1}  \leq l_1(l_1-1)\Big(\frac{2bl_1M^{1/3}}{M}\Big)^2 = \frac{4b^2l_1^3(l_1-1)}{M^{4/3}}.
\]
Now the expected number of steps $0 \leq k \leq M^{1/3}$ for which this happens is at most $C_{5,2}/M$, for some positive constant $C_{5,2} = C_{5,2}(a,b)$, so the desired upper bound on the probability of $\mathbf{A}^c$ follows by Markov inequality.

% The probability that $X_n$ decreases by $a$ in two given consectuive steps is bounded from above by a constant times $(M^{1/4}/M)^2 = M^{-3/2}$, and so by Markov inequality $\p(A) \leq C_{5,2}/M^{5/4}$. 

To handle the last case, by Markov property we can assume that $Y_0 > M^{1/3}$, so that $\tau_2=0$. Consider the process $Y_n'$ such that $Y_0'=Y_0$ and such that at each step $Y_n'$ either decreases by $a$ with probability $\epsilon/(1+\epsilon)$ or increases by $b-a$ with probability $1/(1+\epsilon)$. It is a simple observation that one can couple the processes $Y_n$ and $Y_n'$ so that $Y_n' \leq Y_n$ for all $n \leq \tau_3$. Therefore, defining $\rho'$ as the smallest index $n$ such that $Y_n' \leq Y_0/2$ it is clear that $\rho < \tau_3$ implies $\rho' < \infty$. Now the bound on the probability of $\rho < \tau_3$ follows from 
\[
\p(\rho < \tau_3) \leq \p(\rho' < \infty) \leq e^{-cY_0/2} \leq e^{-cM^{1/3}/2},
\]
for a positive constant $c = c(a,b)$. To justify the second inequality above choose the value of the constant $c>0$ so that 
\[
h(c) = \frac{1}{1+\epsilon}e^{-c(b-a)} + \frac{\epsilon}{1+\epsilon}e^{ca} = 1,
\]
which exist since $h(0)=1$, $\lim_{t \to \infty}h(t)= \infty$ and $h'(0) = -b/(1+\epsilon)+a = - \kappa < 0$. Such a choice then implies that $e^{-cY_n'}$ is a martingale and since $e^{-cY_{n'\wedge \rho'}}$ is bounded, optional stopping theorem implies that 
\[
e^{-cY_0} \geq \p(\rho' < \infty) e^{-cY_0/2},
\]
which yields the inequality.

% Define $n_1 = \lfloor cM\rfloor$, where a positive constant $c=c(a,b)$ is to be chosen below. Observe that for $n \leq n_1$ steps the ratio $Y_n/(M-an)$ is at most
% \[
% \kappa := \frac{\epsilon/(1+\epsilon) + c(b-a) }{1 -ac} \leq \frac{1/3 + c(b-a) }{1 -ac}.
% \] 
% Select the constant $c$ so that $\kappa < 1/2$.
% %which is, by the assumed bound $(6b-3a)c<1$, strictly less than $1/2$.
% Therefore, the probability 

% Consider the process $Y_n'$ such that $Y_0'=Y_0$ and such that at each step $Y_n'$ either decreases by $a$ or increases by $b-a$ with equal probabilities.

% It is a simple observation that we can couple the processes $X_n'$ and $X_n$ so that $X_n' \geq X_n$ for all $n \leq \rho$. However, by Gambler's Ruin for biased random walk (see \cite{Feller}) the probability that $X_n'$ becomes larger than $M - M^{1/4}/2$ before $3M/4$ is less than  $e^{-C_{5,3}M^{1/4}}$ for some positive constant $C_{5,3}>0$ not depending on $M$ and $X_0$. This finishes the argument.
\end{proof}

\begin{proof}[Proof of Theorem \ref{thm:active_edges}, inequality \eqref{eq:concentration in urn models_2}.]
It's clear that by taking $C \geq 1$, it suffices to show the claim  for $t<1$, so we will asume this throughout the proof. 
%Also by rescaling $a$, $b$, $X_0$, $Y_0$ and $M$ we can assume $b\geq 1$ as well.
In the proof below we will assume that $M$ is sufficiently large, so that the presented estimates hold. 
The constant $C$ from \eqref{eq:concentration in urn models_1} will be denoted by $C_0$.
%Furthermore, for any positive constant $k = k(a,b)$, it suffices to prove the claim for $n$ such that $M-an \geq k(a,b)$ (the value of $k(a,b)$ will be chosen later).
%By adjusting the constant $C$ if necessary, it suffices to prove the claim for $t<1$ and when $M-an_t \geq a+b$. While the former claim is obvious, for the second one 
%To see this simply define $t_1$ and $t_2$ as the maxima of the values $t$ so that $M-an_{t} \geq k(a,b)$ and $M-an_{t} \geq 0$ respectively. Then note that for $t > t_2$ the claim holds since $\p(\tau_{X_0,M} \geq n_{t})=0$, and for $t_1< t \leq t_2$ the claim holds be redefining $C$ since $t_2/t_1$ depends only on $a$ and $b$. 

First define $s=t^{-a/(2b-a)}$. Since $t < 1$, we have $s \geq 1 \geq 2C_0X_0^{a/b}M^{-1}$ for $M$ large enough (as $X_0 \leq M$), and by
\eqref{eq:concentration in urn models_1} we have
\begin{equation}\label{eq:back_estimate}\mathbb{P}(\mathbf{K}_{X_0, M, s, 1/2}^c) \leq 4C_0s^{-b/a} = 4C_0t^{b/(2b-a)}.\end{equation}

It is easy to check that the process $X_n / (M-a(n \wedge \rho))$ is a supermartingale: For $\rho \leq n$ the value of the process remains unchanged, and when $\rho > n$ we have
 \begin{align*}
 \mathbb{E}\Big(\frac{X_{n+1}}{M-a((n+1)\wedge \rho)}\Big | \mathcal{F}_n\Big) 
&= \frac{X_n}{M-a(n+1)} - \frac{b}{M-a(n+1)}\frac{X_n}{M-an}\\
& = \frac{X_{n}}{M-a(n\wedge \rho)}\Big(1- \frac{b-a}{M-a(n+1)}\Big).
\end{align*}
Since $n_s \leq n_t$, on the event $\mathbf{K}_{X_0, M, s, 1/2}$ we have
\begin{equation}\label{eq:conditional_estimate}
\e\Big(\frac{X_{n_t}}{M-a(n_t\wedge \rho)}\Big|\mathcal{F}_{n_s}\Big) \leq \frac{X_{n_s}}{M-a(n_s\wedge \rho)} \leq \frac{3X_0}{2M}\Big(1 - \frac{an_s}{M}\Big)^{b/a-1}.
\end{equation}
Observe that if $\tau_{X_0,M} \geq n_t$ then either $\rho \leq n_t$ which in turn implies $X_{n_t} > M-a(n_t\wedge \rho)$ or $\rho > n_t$ and $X_{n_t}\geq b$. In either case we have 
\[
\frac{X_{n_t}}{M-a(n_t\wedge \rho)} \geq  \min\left\{1, \frac{b}{M-an_t}\right\} \geq \frac{b}{M-an_t},
\]
since the assumption on $t$ in \eqref{eq:concentration in urn models_2} implies that $M-an_t \geq b$. Therefore, by Markov inequality
\[
\p(\tau_{X_0,M} \geq n_t) \leq \frac{M-an_t}{b}\e\Big(\frac{X_{n_t}}{M-a(n_t\wedge \rho)}\mathbf{1}_{\mathbf{K}_{X_0,M,s,1/2}}\Big) + \p(\mathbf{K}_{X_0,M,s,1/2}^c),
\]
where $\mathbf{1}_{\mathbf{K}_{X_0,M,s,1/2}}$ is the indicator of the event $\mathbf{K}_{X_0,M,s,1/2}$. Then by \eqref{eq:back_estimate}, \eqref{eq:conditional_estimate} and Remark \ref{rem:Upper_bound_on_stopping-size} we have
\begin{multline*}
\p(\tau_{X_0,M} \geq n_t) \leq \frac{3X_0}{2b}\Big(1-\frac{an_s}{M}\Big)^{b/a-1}\Big(1-\frac{an_t}{M}\Big) + 4C_0t^{b/(2b-a)} \\ \leq \frac{3\cdot 2^{b/a-1}}{b}ts^{b/a-1} +4C_0t^{b/(2b-a)} = \Big(\frac{3\cdot 2^{b/a-1}}{b} + 4C_0\Big)t^{b/(2b-a)},
\end{multline*}
which finishes the proof.

\end{proof}

\begin{proof}[Proof of Corollary \ref{cor:active_edges}] 
This is immediate from Remark \ref{rem:Upper_bound_on_stopping-size} and the fact that the event in the statement is implied by $\mathbf{K}_{X_0,M,t, \varepsilon}$ for $t= 2^{-a/b-1}m^{a/b}$.
\end{proof}

\begin{proof}[Proof of Corollary \ref{cor:estimates_on_stopping}] 
This is immediate from the fact that the event in the statement is implied by $\tau_{X_0,M} \geq n_t$ for $t=X_0^{a/b}/M$.
\end{proof}

\begin{proof}[Proof of Theorem \ref{thm:inactive_edges}]
We will assume that $\epsilon < 1/6$ which is clearly sufficient.
We will use different argument in cases when the $1-an/M$ falls in the intervals
\begin{multline*}
I_1 = \left(0, \frac{1}{X_0^{a/b}\log M}\right], \ \ 
I_2=\left[\frac{1}{X_0^{a/b}\log M}, \frac{\log M}{X_0^{a/b}}\right], \ \ 
I_3=\left[\frac{\log M}{X_0^{a/b}}, \Big(\frac{M}{2X_0}\Big)^{\frac{a}{b-a}}\right],\\
I_4=\left[\Big(\frac{M}{2X_0}\Big)^{\frac{a}{b-a}}, 1- \frac{\log M}{\sqrt{M}}\right], \ \ 
I_5=\left[1-\frac{\log M}{\sqrt{M}},1\right].
\end{multline*}
First we will handle the cases $I_1$, $I_2$ and $I_3$. Then we will handle cases $I_4$, and $I_5$  and for this we will assume that $X_0 \geq M/2$ (that is $Y_0 \leq M/2$), otherwise  $[0,1] \subset I_1 \cup I_2 \cup I_3$ so there is no need to consider the intervals $I_4$ and $I_5$. Note, that even though $X_0 \geq M/2$, the interval $I_4$ might still be empty.
In the case when $Y_0 \leq \sqrt{M}(\log M)^2$ we will further consider two subcases of $I_5$, when $1-an/M$ is in
\[
I_{5,1}=\left[1-\frac{\log M}{\sqrt{M}}, 1-\frac{1}{\sqrt{M}(\log M)^3} \right]\ \text{ and } \ I_{5,2}=\left[ 1-\frac{1}{\sqrt{M}(\log M)^3} , 1\right].
\]

% All of these intervals are non-empty for $M$ large enough, except maybe $I_5$. Therefore when handling the case for which $1-an/M \in I_5$ then we will assume that $I_5$ is nonempty, that is 
% \begin{equation}\label{eq:empty_case}
% X_0 > \frac{M}{2} \Big(1- \frac{\log M}{\sqrt{M}}\Big)^{b/a-1}.
% \end{equation}

 We start by writing
\[
L_n = \frac{M-a(n\wedge \rho)- X_n}{(M-an) - X_0 \Big(1-\frac{an}{M}\Big)^{b/a}},
\]
from where we easily get
\begin{equation}\label{eq: comparison of A and B}
L_n -1 = \frac{ X_0 \Big(1-\frac{an}{M}\Big)^{b/a}-X_n}{(M-an) - X_0
  \Big(1-\frac{an}{M}\Big)^{b/a}} + \xi_n, \ \text{ and }
\ \frac{|K_n-1|}{|L_n-1|} =
\frac{M}{X_0\Big(1-\frac{an}{M}\Big)^{b/a-1}} -1 + \chi_n,
\end{equation}
% First consider all $n$ such that
% \begin{equation}\label{eq: condition on n 1}
% \frac{\log M}{X_0^{a/b}} \leq 1- \frac{an}{M} \leq
% \Big(\frac{M}{2X_0}\Big)^{a/(b-a)}.
% \end{equation}
where $\xi_n$ and $\chi_n$ are equal to zero if $\rho \geq n$.
For all $n$ such that $1-an/M \in I_3$, the right hand side of the second formula above without the $\chi_n$ term is bounded from below by 1. The event  $\mathbf{K}_{X_0,M,\log M, \varepsilon}$ from Theorem \ref{thm:active_edges} implies that $\chi_n=0$ and $|K_n-1| \leq \epsilon$ for all $0 \leq n \leq n_{\log M}$, and in particular for all $n$ such that $1-an/M \in I_3$. 
Therefore, it also implies that $|L_n-1| \leq \varepsilon$ when $1-an/M \in I_3$. 
Moreover, for a given $\varepsilon > 0$ we can choose $M$ large
enough so that $\log M \geq \frac{C}{M^{1-a/b}\varepsilon} \geq
\frac{CX_0^{a/b}}{M\varepsilon}$, where $C$ is the constant from Theorem \ref{thm:active_edges}. Then by \eqref{eq:concentration in urn
  models_1} we have that the probability of the event  $\mathbf{K}_{X_0,M,\log M, \varepsilon}$  is bounded from below by 
\[
1- \frac{C}{(\log
  M)^{b/a}\varepsilon^2},
\]
which resolves the case when $1-an/M \in I_3$. 

Now we look at the case when $1-an/M \in I_2$. 
%Let $\mathbf{R}$ denote the event that the process doesn't end because the number of red balls got negative, that is $\mathbf{R}= \{M-a\rho \leq 0\}$.
By Remark \ref{rem:Upper_bound_on_stopping-size}, $1-an_{\log M}/M \leq 2X_0^{-a/b}\log M$ holds for $M$ large enough, so
 the event $\mathbf{K}_{X_0,M,\log M, \varepsilon}$ implies that
\[
X_{n_{\log M}}  \leq (3\log M)^{b/a}.
\]
Thus, on the event $\mathbf{K}_{X_0,M,\log M, \varepsilon} \cap \mathbf{R}$, 
we have both $0 \leq X_n \leq (3\log M)^{b/a}$ for all $n$ such that $1-an/M \in I_2$ and $\xi_n=0$. From the
first relation in \eqref{eq: comparison of A and B} we get
\[
- \frac{(3\log M)^{b/a}}{\frac{M}{X_0^{a/b}\log M} -(\log
  M)^{b/a} } \leq L_n - 1 \leq \frac{(\log
  M)^{b/a}}{\frac{M}{X_0^{a/b}\log M} -(\log M)^{b/a}}.
\]
The denominator above
is bounded from below by $M^{1-a/b}/\log M - (\log M)^{b/a}$, and so for $M$ large enough both the lower and the upper bound on $L_n-1$ above are smaller than $\epsilon$ in the absolute value. 
Since $\lim_{M \to \infty}\mathbb{P}(\mathbf{K}_{X_0,M,\log M, \varepsilon} \cap \mathbf{R}) = 1$, we
conclude that for any $\varepsilon > 0$ with probability converging to 1
we have that $|L_n-1| \leq \varepsilon$, for all $n$ satisfying $1-an/M \in I_2$.

Next we consider $n$ such that $1-an/M \in I_1$, so in particular $n \geq n_{1/\log M}$. By \eqref{eq:concentration in urn models_2}, with probability of at
least $1- \frac{C}{(\log M)^{b/(2b-a)}}$ we have that $X_n=0$ and $Y_n
= M-an$ for all these $n$. This in particular implies the event $\mathbf{R}$ and therefore, with probability converging to $1$, for all such $n$,
we have that
\[
1 \leq L_n = \frac{1}{1- \frac{X_0}{M}
  \Big(1-\frac{an}{M}\Big)^{b/a-1}} \leq \frac{1}{1-
  \frac{X_0^{a/b}}{M(\log M)^{b/a-1}}} \leq \frac{1}{1-
\frac{1}{M^{1-a/b}(\log M)^{b/a-1}}},
\]
which converges to $1$ as $M \to \infty$.
% The above arguments show that for a fixed $\varepsilon > 0$ with
% probability converging to $1$ (uniformly in $Z_0$) as $M \to \infty$,
% we have that $|L_n-1| \leq \varepsilon$ for all $n$ such that $0 <
% 1-\frac{an}{M} \leq (M/2Z_0)^{a/(b-a)}$. If $Z_0 \leq M/2$ there is
% nothing left to prove, so assume $Z_0 \geq M/2$.
 
We are done with the cases $I_1$, $I_2$ and $I_3$.
As we said we now assume $X_0 \geq M/2$. Consider the case when $1-an/M \in I_4$.
% Now consider the case when
% \begin{equation}\label{eq: condition on n 3}
% \Big(\frac{M}{2Z_0}\Big)^{a/(b-a)} \leq 1- \frac{an}{M} \leq 1-
% \frac{\log M}{\sqrt{M}}.
% \end{equation}
By the second relation in \eqref{eq: comparison of A and B} the
condition that both
\begin{equation}\label{eq: condition on A 3}
|K_n-1| \leq \varepsilon\Big(MX_0^{-1}\Big(1-\frac{\log
  M}{\sqrt{M}}\Big)^{-b/a+1}-1\Big)
\end{equation}
and $\rho > n$
hold when $1-an/M \in I_4$ 
implies that $|L_n-1| \leq \varepsilon$ holds for these $n$ as well. Denote the right hand side of \eqref{eq: condition on A 3} by $\delta_M$. In particular the event $\mathbf{K}_{X_0,M,t,\delta_M}$, for $t= (MX_0^{-a/b}/2)^{a/(b-a)}$ will imply that $|L_n-1| \leq \varepsilon$ for $1-an/M \in I_4$.
To calculate the probability of $\mathbf{K}_{X_0,M,t,\delta_M}$ we can apply \eqref{eq:concentration in urn
  models_1}. To justify this application we need to check that, for $M$ large enough and $\varepsilon < 1/6$ fixed, $\delta_M \leq 1/2$ and $t \geq \frac{CX_0^{a/b}}{M\delta_M}$. 
Both will follow if we prove that
\begin{equation}\label{eq:conditions_on_delta_M}
\frac{2^{a/(b-a)}C}{M} \leq \delta_M \leq 1/2.
\end{equation}
Since we assumed that $X_0 \geq M/2$ it is easy to show that
\begin{equation}\label{eq:bounds_on_delta_M}
\epsilon \frac{b-a}{a} \frac{\log M}{\sqrt{M}} \leq \delta_M \leq 2\epsilon \Big(1-\frac{\log
  M}{\sqrt{M}}\Big)^{-b/a+1}.
\end{equation}
These inequalities now immediately imply the ones in \eqref{eq:conditions_on_delta_M} for a fixed $\epsilon < 1/6$ and $M \geq M(\epsilon)$ large enough.
% The former follows from the fact that $\epsilon \leq 1/2$ and $X_0 \geq M/2$. The latter is equivalent to
% \[
% \delta_M \geq C(2X_0)^{a/(b-a)}M^{-b/(b-a)},
% \]
% which holds since the right hand side is at most $2^{a/(b-a)}CM^{-1}$ and the left hand side is at least $\epsilon C_1M^{-1/2}\log M$, for some constant $C_1=C_1(a,b)$ which depends only on $a$ and $b$. 
Therefore we can apply the estimate in \eqref{eq:concentration in urn models_1}, and by the lower bound in \eqref{eq:bounds_on_delta_M}, the probability that $|L_n-1| \leq \varepsilon$ for all $n$ such that $1-an/M \in I_4$ is at least
\[
 1- \frac{C}{t^{b/a}\delta_M^2} 
 \geq 1- \frac{2^{b/(b-a)}a^2CX_0^{a/(b-a)}M}{\varepsilon^2 (b-a)^2M^{b/(b-a)}(\log M)^2}
 \geq 1- \frac{2^{b/(b-a)}a^2C}{\varepsilon^2 (b-a)^2(\log M)^2}.
\]
% where in the last line we used the inequality $(1-x)^{-\gamma} \geq 1+
% \gamma x$, for all $\gamma > 0 $ and $0 < x < 1$, with $\gamma = b/a$
% and $x = \log M / \sqrt{M}$.
Since the right hand side %of \eqref{eq: condition on n 3 -estimate 1}
converges to $1$ when $\epsilon$ is fixed and $M \to \infty$, we are only left to consider the  final case when $1-an/M \in I_5$. To this end write
\begin{equation}\label{eq:new formula for L}
L_n = \frac{Y_n}{\Big(1-\frac{an}{M}\Big)\Big(Y_0 \Big(1-
  \frac{an}{M}\Big)^{b/a-1} + M\Big(1- \Big(1-
  \frac{an}{M}\Big)^{b/a-1}\Big)\Big)}.
\end{equation}
First consider the case when $Y_0 \geq \sqrt{M}(\log M)^2$ so that
$Y_0(1-\varepsilon/3) \leq Y_n \leq Y_0(1+\varepsilon/3)$, for $M$ large
enough and when $1-an/M \in I_5$. Furthermore, for such $n$ and $M$ large enough the denominator satisfies
\begin{multline*}
\Big(1 - \frac{\log M}{\sqrt{M}}\Big)^{b/a}Y_0 \leq
\Big(1-\frac{an}{M}\Big)\Big(Y_0 \Big(1- \frac{an}{M}\Big)^{b/a-1} +
M\Big(1- \Big(1- \frac{an}{M}\Big)^{b/a-1}\Big)\Big) \\ \leq Y_0 +
2(b-a)n \leq Y_0 + 2\frac{b-a}{a}\sqrt{M}\log M,
\end{multline*}
and so it is bounded by $Y_0(1-\varepsilon/3)$ from below and by
$Y_0(1+\varepsilon/3)$ from above, for $M$ large enough. This now implies
the deterministic fact that $|L_n-1| \leq \varepsilon$ when $1-an/M \in I_5$.

Now assume $Y_0 \leq \sqrt{M}(\log M)^2$. Clearly, for $M$ large enough, and all $n$ which satisfy $1-an/M \in I_5$ we have $Y_n \leq 2\sqrt{M} (\log M)^2$, and in particular the probability of drawing a red ball in this stage is at most $3(\log M)^2/\sqrt{M}$, for $M$ large enough. Therefore, the expected number of red balls drawn in the stage when $1-an/M \in I_{5,1}$ is no more than $3(\log M)^3/a$, and in the the stage when $1-an/M \in I_{5,2}$ is no more than $3 (a\log M)^{-1}$. By Markov inequality, with probability converging to 1 (as $M\to \infty$) we have both
\begin{equation}\label{eq:stage_5_1_numerator}
Y_n = Y_0 + (b-a)n, \ \text{ whenever } 1-an/M \in I_{5,2}
\end{equation}
and 
\begin{equation}\label{eq:stage_5_2_numerator}
Y_0 + (b-a)n- (\log M)^4 \leq Y_n \leq Y_0 + (b-a)n, \ \text{ whenever } 1-an/M \in I_{5,1}.
\end{equation}
(Note that if \eqref{eq:stage_5_1_numerator} holds then \eqref{eq:stage_5_2_numerator} implies that the value of $Y_n$ is positive throughout the $I_5$ phase.) Moreover, for a fixed $\epsilon > 0$, $M$ large enough, and any $n$ such that $1-an/M \in I_5$ 
\begin{equation}\label{eq: condition on n 4 - help}
\frac{1}{1+\varepsilon/2} \leq \Big(1- \frac{an}{M}\Big)^{b/a}\leq
1,\ \text{and} \ \frac{(b-a)n}{1+\varepsilon/2}\leq
M\Big(1- \Big(1-\frac{an}{M}\Big)^{b/a-1}\Big) \leq
\frac{(b-a)n}{1-\varepsilon/2}.
\end{equation}
In the view of \eqref{eq:new formula for L}, the case $I_{5,2}$ follows from \eqref{eq:stage_5_1_numerator} and \eqref{eq: condition on n 4 - help}. For the case $I_{5,1}$ we just need to argue about the lower bound and for this it suffices to show that for any $1-an/M \in I_{5,1}$
\[
\frac{(\log M)^4}{Y_0 + (b-a)n} < \epsilon/2,
\]
for $M$ large enough which surely holds since for any $n \in I_{5,1}$ we have $n \geq \frac{\sqrt{M}}{a(\log M)^3}$.

\end{proof}
%}

%\newpage

% \bibliography{mybib}
% \bibliographystyle{plain}

% \end{document}

%  LocalWords:  dR  thm Equilibrea Bharathi Kempe Salek dN k'M Doob's
%  LocalWords:  automata Sudbury Ising submodular Bollob tuple nY MX
%  LocalWords:   matchings dt lya rr eq bs const integrable
%  LocalWords:   observables KALLENBERG stochastically Hoeffding's MZ
%  LocalWords:  BENAIM kq kI tY CX c'L monotonicity th piecewise dr
%  LocalWords:  Poissonized JANSON kY summand tf  ds  subintegral U'V
%  LocalWords:  Svante Janson Chebyshev janson V'U VU nZ dk dL dx dn
%  LocalWords:  nk Markov's Durrett hipercubes tA Voronoi tZ bZ ak kp
%  LocalWords:  supermartingale C'n CFPP MCFPP CP CFP functio choise
% LocalWords:  boud ps nonnegative secon rn markov
% \bibliography{references}

\begin{thebibliography}{1}

\bibitem{main}
Ton\'ci Antunovi\'c, Yael Dekel, Elchanan Mossel, and Yuval Peres.
\newblock Competing first passage percolation on random regular graphs.
\newblock preprint, available at arXiv:1109.4918 [math.PR].

\bibitem{BenderCanfield}
Edward~A. Bender and E.~Rodney Canfield.
\newblock The asymptotic number of labeled graphs with given degree sequences.
\newblock {\em J. Combinatorial Theory Ser. A}, 24(3):296--307, 1978.

\bibitem{Bollobas80}
B.~Bollob\'{a}s.
\newblock A probabilistic proof of an asymptotic formula for the number of
  labelled regular graphs.
\newblock {\em European J. Combin.}, 1(4):311--316, 1980.

\bibitem{Bollobas01}
B{\'e}la Bollob{\'a}s.
\newblock {\em Random graphs}, volume~73 of {\em Cambridge Studies in Advanced
  Mathematics}.
\newblock Cambridge University Press, Cambridge, second edition, 2001.

\bibitem{JLR00}
Svante Janson, Tomasz {\L}uczak, and Andrzej Rucinski.
\newblock {\em Random graphs}.
\newblock Wiley-Interscience Series in Discrete Mathematics and Optimization.
  Wiley-Interscience, New York, 2000.

\end{thebibliography}
% \bibliographystyle{plain}

\section*{Acknowledgments}
This work was a part of a project, most of which was done in collaboration with Yael Dekel, Elchanan Mossel and Yuval Peres. The author would like to thank them for fruitful discussions, encouragement and support.

\end{document}